\documentclass[10pt,final,reqno]{amsart}
\usepackage{latexsym,amsmath,amstext,amssymb,amsopn,amsthm,verbatim,euscript,amscd,mathrsfs}
\usepackage{otitis}
\usepackage{showkeys}

\DeclareMathAlphabet{\mathpzc}{OT1}{pzc}{m}{it}

\newcommand{\bis}{\mathscr{B}}

\newcommand{\R}{\mathbb{R}}

\newcommand{\calA}{\mathcal{A}}

\newcommand{\calG}{\mathcal{G}}

\newcommand{\calU}{\mathcal{U}}
\newcommand{\calY}{\mathcal{Y}}

\newcommand{\calX}{\mathcal{X}}

\newcommand{\dd}{{\mathrm{d}}}  

\newcommand{\PsM}{\mathrm{P}\sigma\mathrm{M}}

\newcommand{\target}{t}

\newcommand{\Lie}[2]{{\left[{\,{#1}\,,\,{#2}\,}\right]}} 
\newcommand{\braket}[2]{\left\langle{\,{#1}\,,\,{#2}\,}\right\rangle} 

\newcommand{\vecc}[1]{{\stackrel{\rightarrow}{#1}}}

\newcommand{\lie}[2]{[#1,#2]}

\newtheorem{Thm}{Theorem}[section]
\newtheorem{Prop}[Thm]{Proposition}
\newtheorem{Lem}[Thm]{Lemma}
\newtheorem{Cor}[Thm]{Corollary}

\newtheorem*{Thm*}{Theorem}
\newtheorem*{Lem*}{Lemma}

\newtheorem{Def}[Thm]{Definition}

\theoremstyle{remark}
\newtheorem{Rem}[Thm]{Remark}
\newtheorem*{Ack}{Acknowledgment}

\newtheorem*{Rem*}{Remark}

\theoremstyle{definition}

\begin{document}
\title{On Nonlinear Gauge Theories}

\author[D.~Signori]{Daniele Signori}
\author[M.~Sti\'enon]{Mathieu Sti\'enon}

\address{Department of Mathematics, Pennsylvania State University, University Park, PA 16802, United States}
\email{signori@math.psu.edu stienon@math.psu.edu}


\begin{abstract}
In this note, we study non-linear gauge theories for principal bundles, where the structure group is replaced by a Lie groupoid. We follow the approach of 
\moerdijk and establish its relation with the existing physics literature. In particular, we derive a new formula for the gauge transformation which closely resembles and generalizes the classical formulas found in Yang Mills gauge theories.
\end{abstract}
\maketitle


\section{introduction}
Differential geometry of principal bundles is the natural language of gauge theories of Yang-Mills type. Since the late 50's this recognition spawned many fruitful ideas in both mathematical and physical sciences, ranging from the discovery of new topological invariants to the formulation of the Standard Model. The object of the present study are principal groupoid bundles, where the structure Lie group is replaced in a suitable way by a Lie groupoid. 

A nonlinear gauge theory is a model where the symmetry content is described by field dependent quantities which replace the Lie algebra structure constants. The language of Lie algebroids and groupoids provides a natural setting for this generalization \cite{Strobl:2004}. An example is the Poisson Sigma Model ($\PsM$), a two dimensional topological field theory whose path integral quantization  \cite{cattaneo.felder:2000a} yields Kontsevich's formula for deformation quantization \cite{kontsevich:2003a}. 
Let $\Sigma$ be a two dimensional surface (bounded, oriented), $(M,\pi)$ a Poisson manifold: the $\PsM$ action is a functional on the space of vector bundle morphisms $T\Sigma \to T^*M$. 
Such maps decompose into pairs $(X,\eta)$, where $X$ is a map between the basis and $\eta \in \Omega^1(\Sigma, X^*T^*M)$.
The equations of motion for this classical theory are the conditions defining Lie algebroid morphisms from the standard Lie algebroid $T\Sigma$ to the cotangent Lie algebroid $T^*M$, whose Lie algebroid structure is defined by $\pi$. 

Such a Poisson $\sigma$-model leads naturally to a construction of symplectic groupoid of $M$ as shown by Cattaneo-Felder \cite{cattaneo.felder:2001c}.
For an arbitray Lie algebroid $A$, a similar construction leads to the Lie groupoid integrating $A$, which was extensively studied by Crainic-Fernandes leading to the solution of the problem of Lie's third theorem for Lie algebroids \cite{crainic.fernandes:2003}.

On the other hand, for a principal groupoid bundle, there is a natural notion of connections, namely horizontal distributions. Indeed, Moerdijk-Mr\v{c}un  \cite{moerdijk.mrcun:2003a}introduced a definition of connections which reduces to the standard connections for usual princial bundles (with structure group being a Lie group).

The purpose of this paper is to establish relations between the approach of Moerdijk-Mr\v{c}un and the one commonly adopted in non-linear gauge theory. In particular, we give an intepreation of gauge transformations in non-linear gauge theory in terms of principal groupoid bundles, and derive a new 	formula resembling the classical formula for the effect of gauge transformations on gauge fields \cite{bonechi.zabzine:2005a, Ikeda:1994, schaller.strobl:1994a, cattaneo.felder:2001c}. 

The paper is organized as follows.
In Section~\ref{sec.PGB}, we recall the notion of principal bundle $P \to M$ with structure groupoid $G\toto G_{0}$. 
Principal connections on a PGB are introduced in Section~\ref{sec.Connections}, as horizontal distributions $H$ on the fiber bundle $P$ satisfying suitable equivariance properties. 
Connections are then characterized as equivariant bundle maps $TP\to A$ satisfying the \moerdijk condition. Their curvature measures how far the corresponding bundle maps are from being Lie algebroid morphisms.
Next, suitable gauge transformations are introduced.
Finally, we discuss the case of trivial principal bundles.


\begin{Ack}
The authors are grateful to Martin Bojowald and Ping Xu for sharing their insight and useful comments. 
The authors acknowledge the hospitality of Institut Henri Poincar\'e, ETH Z\"urich and Universit\"at Z\"urich, where part of this work was done.
This research was partially supported by the European Union through the FP6 Marie Curie RTN ENIGMA 
(contract number MRTN-CT-2004-5652).
\end{Ack}


\section{Principal groupoid bundles}\label{sec.PGB}

In this section we introduce principal groupoid bundles and fix the notation, we mainly follow \cite{moerdijk.mrcun:2003a,rossi:2004b}. For basic material on Lie groupoids and algebroids we refer to \cite{mackenzie:2005a, cannasdasilva.weinstein:1999a}. A very detailed account on principal groupoid bundles (PGB) together with a comprehensive bibliography can be found in \cite{rossi:2004b}. 

\subsection{Principal groupoid bundles}

For a groupoid $G$, $G_{0}$ denotes its set of objects, $s,t: G \to G_0$ are the \emph{source} and \emph{target} maps, $1:G_0\to G$ is the unit map, the composition of arrows is defined for {\emph{composable pairs},} i.e. $(g,h) \in G\times G$ such that $t(g) = s(h)$. The group of local bisections is denoted by $\bis(G)$. Given $\gamma \in \bis(G)$ a \emph{left (right) translation} is indicated with $L_\gamma$ ($R_{\gamma}$). For source ($s$) and target ($t$) fibers we use the notations $G_x := s^{-1}(x)$ and $G^y:= t^{-1}(y)$.

A right action of a groupoid $G$ on a manifold $E$ along the moment map $\epsilon:E \to G_0$ is a smooth map:
\[ \mu: E{_{\epsilon}\times_{s}}G\to E 
:(p,g) \mapsto pg \; ,\]
such that, whenever defined, the following identities hold: $\epsilon(pg) = t(g)$, $p \cdot 1_{\epsilon(p)} = p$ and $p(gh) = (pg)h$.
Here, $E{_{\epsilon}\times_{s}}G$ denotes the fibered product $\{(p,g) \in E\times G| \epsilon(p) = s(g)\}$.

The Lie algebroid $(A,\Lie{}{})$ of the Lie groupoid $G$ has anchor map $\rho = t_{*}$.

\begin{Def}
A \emph{principal groupoid bundle} $P$ over a manifold $M$ is a smooth fiber bundle $\pi:P\to M$ equipped with a smooth right action $\mu$ along $\epsilon:P \to G_{0}$ which is fiberwise, 
such that the map:
\[ (\pr_1,\mu): P{_{\epsilon}\times_{s}}G \to 
P{_{\pi}\times_{\pi}} P : (p,g) \mapsto (p, pg) \; ,\]
is a diffeomorphism.
\end{Def} 
A principal groupoid bundle 
\begin{equation*} 
\xymatrix{ P \ar[dr]^{\epsilon} \ar[d]_{\pi} 
& G \ar@<2pt>[d]\ar@<-2pt>[d] \\ M & G_0 } 
\end{equation*}
is sometimes referred to as $(M\xfrom{\pi}P\circlearrowleft_{\epsilon} G)$.

\subsection{Division map}

The projection onto $G$ of the inverse map of $(\pr_1,\mu)$ is called \emph{division map}, $\delta_{P}$.
\[ \delta_{P}: P{_{\pi}\times_{\pi}}P \to G \; .\]
In words, the right multiplication by the arrow $\delta_{P}(p,q)$ sends $p$ to $q$: 
\[ p\cdot\delta_{P}(p,q)=q .\] 
Some properties:

\begin{Prop}\label{prop.DivisionMap}
For all $(p,q) \in P {_{\pi}\times_{\pi}} P$, the division map $\delta_{P}$  satisfies the following properties:
\begin{enumerate}
\item $\delta_{P}(p,p) = 1_{\epsilon(p)}$;
\item $(s \circ \delta_{P})(p,q) = \epsilon(p)$, \, $(t \circ \delta_{P})(p,q) = \epsilon(q)$;
\item $\delta_{P}(p,q)^{-1} = \delta_{P}(q,p)$;
\item $\delta_{P}(p,qg) = \delta_{P}(p,q)g$;
\end{enumerate}
\end{Prop}

When there is no risk of confusion we write $\delta$ for $\delta_{P}$.

\subsection{Trivializations}\label{sec.Trivializations}

A groupoid $G$ is itself a principal groupoid bundle 
$(G_0\xfrom{s}G\circlearrowleft_{t} G)$, the right action is given by the groupoid multiplication. It is usually denoted with $\calU_{G}$ and called \emph{unit bundle of the Lie groupoid $G$}.
The division map $\delta_{\calU_{G}}$ is given by:
\[ G {_{s}\times_{t}} G \ni (g,h) \mapsto g^{-1}h \in G \; ,\]
hence the terminology.

We define the \emph{trivial $G$-bundle over M w.r.t. a given map $M\overset{f}\to G_{0}$} as the pullback bundle 
\[ f^{*}\calU_{G} = \{ (m,g) \in M\times G | f(m)=s(g) \} \] endowed with the momentum map 
\[ \epsilon:f^{*}\calU_{G}\to G:(m,g)\mapsto t(g)\] and the action 
\[ \mu:f^{*}\calU_{G} {_\epsilon\times_s} G :\big((m,g_1),g_2\big)\mapsto (m,g_1 g_2) .\]

\begin{Prop}\label{LocalTrivializations}
Every principal $G$-bundle $(M\xfrom{\pi}P\circlearrowleft_{\epsilon} G)$ is locally diffeomorphic to a trivial bundle.
\end{Prop}

\begin{Rem}
Contrary to the classical case, for a principal groupoid bundle a trivialization includes as a priori data a (local) section $\sigma_{i}$, which is used to specify the map $\epsilon_{i}$ along which the unit bundle $\calU_{G}$ is pulled back. Indeed, a section completely defines a trivialization.
\end{Rem}

\subsection{Transition maps}

Let $U_{i}, U_{j} \subset M$ be open subsets with nonempty intersection $U_{ij} = U_{i} \cap U_{j}$. Let $\Phi_{i}$ and $\Phi_{j}$ be the local trivializations corresponding to the sections $U_{i} \overset{\sigma_{i}}\to P$ and $U_{j}\overset{\sigma_{j}}\to P$. Then $\Phi_{ji}:= \Phi_{j}^{-1} \circ \Phi_{i}$:
\begin{equation*}
\Phi_{ji}: (\epsilon^{*}_{i} \calU_{G})_{|{U_{ij}}} \to (\epsilon^{*}_{j} \calU_{G})_{|{U_{ij}}} : 
(m,g) \mapsto (m, \delta_{P}(\sigma_{j}(m), \sigma_{i}(m))g) \;,\end{equation*}
is a diffeomorphism. 

The \emph{transition map} from $U_{i}$ to $U_{j}$ is the map:
\begin{equation*}
\varphi_{ji}: U_{ij} \to G :
m \mapsto \delta_{P}(\sigma_{j}(m), \sigma_{i}(m)) 
\; .\end{equation*}
For an open cover $\{U_{i}\}$ of $M$, $\{\varphi_{i}\}$ is a non abelian 1-cocycle, i.e. it satisfies the identity $\varphi_{ij} \varphi_{jk} = \varphi_{ik}$, on the triple intersection $U_{ijk}$. It is a standard result that PGBs over $M$ are classified up to isomorphism by the non abelian cohomology group $H^{1}(M,G)$.

Given a section $U_{j} \overset{\sigma_{j}} \to P$ and the transition map $\varphi_{ji}$, the section 
$U_{ij}\xto{\sigma_{i}} P$ is given for all $m$ in $U_{ij}$ by:
\[ 
\sigma_{i}(m) = \sigma_{j}(m) \varphi_{ji}(m) \; .
\]



\section{Horizontal distribution on principal groupoid bundles}\label{sec.Connections}

\subsection{Connections}\label{connections}

Recall that a \emph{horizontal distribution}
on a fiber bundle $P \overset{\pi}{\to} M$ is
 a subbundle $H \overset{\imath}{\hookrightarrow} TP$ such that $TP = \ker \pi_* \oplus H$.

Let $\mpg$ be a principal groupoid bundle. A horizontal distribution defines a \emph{connection 1-form} $\alpha \in \Omega^1(P,\epsilon^{*}A)$, which is obtained by assigning  to each vector its vertical component.


To construct explicitly the connection 1-form, one can use the division map $\delta$. When restricted to the second coordinate for a fixed $p \in P$, it gives rise to a diffeomorphism
$ \delta(p, \cdot) := \delta_{p}: P_{\pi(p)} \to G_{\epsilon(p)}$.
Let $X_p \in T_p P$, then the connection form 
$\alpha \in \Omega^{1}(P,\epsilon^{*}A)$ associated to the distribution $H$ is:
\[ \alpha(X_p)=(\delta_{p})_{*,p}\Ver(X_p) \; ,\]
where $\Ver(X_p)$ is the projection of $X_p$ on $\ker \pi_*$ uniquely defined by the splitting $TP = H\oplus\ker\pi_{*}$. Equivalently $\alpha$ can be thought of as a bundle map $TP \to A$ over the momentum map $\epsilon:P \to G_0$:

\begin{Lem}\label{ab}
Let $\mpg$ be a principal groupoid bundle.
There is one-one correspondence between  the following:
\begin{enumerate}
\item horizontal distributions $H\subset TP$;
\item 1-forms $\alpha \in \Omega^{1}(P,\epsilon^{*}A)$;
\item bundle maps $\alpha\diese: TP \to A$ over the momentum map $\epsilon:P \to G_0$.
\end{enumerate}
\end{Lem}

We say that a horizontal distribution satisfies the \emph{\moerdijk condition} if $\epsilon_* H = 0$ \cite{moerdijk.mrcun:2003a}. In this case the tangent action
of $G$ is well defined.

\begin{Def}
A \emph{principal connection} on a principal groupoid bundle $\mpg$ is a horizontal distribution satisfying the \moerdijk condition and such that:
\[ H_{p.g} = (R_g)_* H_p \;, \quad \forall (p,g) \in P {_{\epsilon}\times_{s}} G \;.\]
\end{Def}


\begin{Prop}
Let $\mpg$ be a principal groupoid bundle. Then
\begin{enumerate}
\item A horizontal distribution $H$ satisfies the \moerdijk condition if and only if $\alpha\diese: TP \to A$ as defined in Lemma~\ref{ab} is a bundle map over $\epsilon$ such that 
\begin{equation*}\label{AnchorConditionGlobal}
\rho \circ \alpha^{\sharp} = \epsilon_{*} 
\;.\end{equation*}
\item $H$ is a principal connection if and only if,
in addition, the connection 1-form $\alpha$ is equivariant, i.e. for any local bisection $\gamma$ of $G$:
\[ R_{\gamma}^{*} \alpha = \Ad_{\gamma^{-1}}\alpha \; .\]
\end{enumerate}
\end{Prop}

\begin{proof}
(a) Let $p\in P$ and $X$ be an arbitrary vector field on $P$. $\Ver(X_{p})$ can be written as:
\[ 
\Ver(X_{p}) = \frac{d}{d\tau} \Big|_{0} \Big( p \cdot g_{\tau}\Big) \;,
\] 
where $\tau\mapsto g_{\tau}$ is a path in $G$ such that $s(g_{\tau})=\epsilon(p)$ and $g_{0} = 1_{\epsilon(p)}$. Notice that
\begin{align*}
\alpha(X_{p}) 
=&  (\delta_{p})_{*,p} \Ver(X_p) \\
=&  (\delta_{p})_{*,p} \frac{d}{d\tau}\Big|_{0} \Big( p \cdot g_{\tau}\Big) \\
=&  \frac{d}{d\tau}\Big|_{0} \delta(p,p\cdot g_{\tau}) \\
=&  \frac{d}{d\tau}\Big|_{0} \Big(\delta(p,p)\cdot g_{\tau} \Big) = 
\frac{d}{d\tau}\Big|_{0} g_{\tau} \;.
\end{align*}
Thus the horizontal part of $X_{p}$ is
\[
\Hor(X_{p}) = X_{p} -  \frac{d}{d\tau} \Big|_{0}\Big( p \cdot g_{\tau}\Big) \;,
\]
and one can apply $\epsilon_{*}:TP \to TG_{0}$ to both sides to obtain:
\begin{align*}
\epsilon_{*} \Hor(X_{p}) =& 
\epsilon_{*} X_{p} - \epsilon_{*} \frac{d}{d\tau} \Big(p \cdot g_{\tau}\Big)\Big|_{0} \\
=& \epsilon_* X_{p} - \frac{d}{d\tau} \epsilon(p\cdot g_{\tau})\Big|_{0} \\
=& \epsilon_* X_{p} -  \frac{d}{d\tau} t(g_{\tau}) \Big|_{0} \\
=& \epsilon_* X_{p} - \rho \circ \alpha(X_{p})
\end{align*}
Therefore:
\[
\epsilon_{*} H = 0 \Longleftrightarrow \epsilon_{*} =\rho \circ \alpha \;.
\]


(b) This is obvious.
\end{proof}

\subsection{Curvature}

We denote the horizontal and vertical
parts of a vector $v\in TP$ by $\Hor(v)$ and $\Ver(v)$
respectively. Recall that the 
Ehresmann 
curvature of a 
horizontal distribution $H$, on the bundle $P\xrightarrow{\phi}M$, is 
the $2$-form on $P$, valued
 in the vertical space $\ker\phi_*$, which is  defined by
\begin{equation} \label{eq:defehrcurv}
\omega(u,v)=-\Ver\big(\lie{\Hor(\tilde{u})}{\Hor(\tilde{v})}\big),
 \end{equation}
 where $u,v\in T_p P$ and $\tilde{u}, \tilde{v}$ are vector fields on $P$
 such that $\tilde{u}_p=u$ and $\tilde{v}_p=v$.
It is easy to check that the right hand side of Equation~\eqref{eq:defehrcurv}
is well defined, i.e. independent of the choice of the vector fields $\tilde{u},\tilde{v}$.

Using the action of the groupoid $G$, one can identify the vertical space $\ker\phi_*\subset T_p P$ with $A|_{\epsilon p}$. Therefore, the Ehresman curvature can be seen as a 2-form
$\Omega\in\Omega^{2}(P,\epsilon^{*}A)$, or equivalently
a bundle map 
\[ \Omega^\#:\wedge^2TP\to A \]
over $\epsilon:P\to G_0$.

Given $\alpha\in\OO(P,\epsilon^* A)$, let 
\[ F_{\alpha} = \dd_{\mathrm{P}}\circ{\alpha\diese}^{*} 
-{\alpha\diese}^{*}\circ\dd_{A} :\Omega^1(A)\to \Omega^2(P) ,\]
where $\dd_{\mathrm{P}}:\OO^1(P)\to\OO^2(P)$ and 
$\dd_{A}:\sections{A^*}\to\sections{\wedge^2 A^*}$ are the differentials of the Lie algebroids $TP$ and $A$ respectively.
The following result relates the curvature to $F_{\alpha}$.

\begin{Prop}\label{cd}
Assume that $H$ is a horizontal distribution
satisfying the \moerdijk condition with curvature $\Omega$. Then
\[ (\Omega^\#)^* =F_{\alpha} ,\]
where $\alpha\in\Omega^1(P,\epsilon^* A)$ is as in Lemma~\ref{ab} and both sides are considered as maps $\Omega^1(A)\to\Omega^2(P)$.
\end{Prop}

\begin{proof}
Consider $X_{p} \in T_{p} P$. Take its image $\alpha(X_{p})$ and extend it to a smooth section $\calX$ of $A$. Let $\vecc \calX \in \Gamma(VP)$ be the fundamental vector field on P generated by $\calX$. 
Thus $\alpha\diese(\vecc\calX) =\calX_{\epsilon(p)}$.
Take $U\subset P$ such that there is a smooth horizontal extension $\widetilde{\Hor X_{p}} \in \Gamma(H_{|U})$ of $\Hor X_{p}$. Choose an open neighbourhood $V$ of $p$ in $P$ with $V \subset \cc{V} \subset U$, and a function $f \in C^{\infty}(P,\R)$ such that $\supp(f) \subset U$ and $f_{|V}=1$. Then the vector field
\[ X = \vecc \calX + f.\widetilde{\Hor X_{p}} \;\in \Gamma(TP) \]
is $\alpha$-related to $\calX$.
Similarly, given $Y_{p} \in T_{p} P$ we construct a pair of $\alpha$-related vector fields $Y$ and $\calY$ on an open neighbourhood of $p$, such that the vector field $Y$ extends the tangent vector $Y(p)$. 
For any $\lambda \in A^{*}$, one has:
\begin{align*}
(F_{\alpha}\cdot\lambda) (X \wedge Y)_{p} 
=& \braket{(\alpha\diese)^{*}d_{A} \lambda}{X_{p} \wedge Y_{p}} 
- \braket{d_P((\alpha\diese)^{*}.\lambda)}
{X_{p} \wedge Y_{p}} \\
=& \braket{d_{A}\lambda}{\alpha\diese.X_{p} \wedge \alpha\diese.Y_{p}} 
- \braket{d_P((\alpha\diese)^{*}.\lambda)}
{X_{p} \wedge Y_{p}} \\
=& \braket{d_{A}\lambda}{\calX \wedge \calY}_{\epsilon(p)} 
- \braket{d_P((\alpha\diese)^{*}.\lambda)}
{X_{p} \wedge Y_{p}} \\
=& \rho(\calX)_{\epsilon(p)} \braket{\lambda}{\calY} 
- \rho(\calY)_{\epsilon(p)} \braket{\lambda}{\calX} 
- \braket{\lambda}
{\Lie{\calX}{\calY}_A}_{\epsilon(p)} \\
&- X_{p}\braket{(\alpha\diese)^{*}.\lambda}{Y} 
+ Y_{p}\braket{(\alpha\diese)^{*}\lambda}{X} 
+ \braket{(\alpha\diese)^{*}.\lambda}
{\Lie{X}{Y}_{p}} \\
=& \braket{\lambda}{\alpha.\Lie{X}{Y}_{p} 
- \Lie{\calX}{\calY}_A|_{\epsilon(p)}} \;.
\end{align*} 
The last equality follows from the \moerdijk condition and from 
\[ \epsilon^{*}\braket{\lambda}{\calX} = \braket{\alpha^{*}\lambda}{X} \;.\]
We conclude that
\[ (F_{\alpha}.\lambda) (X \wedge Y)_{p} = \braket{\lambda}{\alpha\diese.\Lie{X}{Y}_{m} - \Lie{\calX}{\calY}_{A}|_{\epsilon(p)}} \;.\]
Finally, we have 
\begin{multline*}
\Lie{X}{Y}=\Lie{\vecc \calX + f.\widetilde{\Hor X_{p}}}{\vecc \calY + g.\widetilde{\Hor Y_{p}}} \\ 
=\Lie{\vecc \calX}{\vecc \calY}
+\vecc \calX(g)\widetilde{\Hor Y_{p}}
-\vecc \calY(f)\widetilde{\Hor X_{p}}
+\Lie{f.\widetilde{\Hor X_{p}}}{g.\widetilde{\Hor Y_{p}}}
\end{multline*}
and 
\[ \alpha\Lie{X}{Y}= 
\alpha\Lie{\vecc \calX}{\vecc \calY} 
+\alpha\Lie{f.\widetilde{\Hor X_{p}}}{g.\widetilde{\Hor Y_{p}}} 
=\Lie{\calX}{\calY}_A+\Omega(X,Y) .\]
\end{proof}

Recall that a bundle map between Lie algebroids $f: B \to C$ is a Lie algebroid morphism if and only if its transpose is a chain map $f^{*}:(\sections{\wedge^*C},\dd_{C})\to(\sections{\wedge^*B},\dd_{B})$. 

\begin{Cor}
As above, let $H$ be a horizontal distribution.
Then $H$ is a flat connection satisfying the \moerdijk condition if and only if $\alpha\diese: TP\to A$ is a Lie algebroid morphism. 

$H$ is a flat principal connection if and only if, in addition, $\alpha$ is equivariant. 
\end{Cor}
 


\subsection{Gauge transformations}
\label{subsec.Moerdijk}

The next step is the introduction of a suitable concept of \emph{gauge transformation}. 
Suppose now $P$ is trivialized by a section $\sigma:M\to P$. Define the pullback connection 1-form $\theta_\sigma := \sigma^*\alpha
\in \Omega^1 (M, \epsilon_\sigma^* A)$, which is a $\epsilon_\sigma^* A$-valued
one-form on $M$. And its induced bundle map $\theta\diese_\sigma:
TM\to A$ over $\epsilon_\sigma (:= \epsilon \circ \sigma): M\to G_0$ is simply the composition
of $\sigma_*: TM\to TP$ with $\alpha\diese: TP\to A$.


\[
\xymatrix@!0{
& TP \ar@{-}[d] \ar[rrd]^{\alpha\diese} & &   \\
TM\ar@/^/[ur]^{\sigma_{*}}\ar[rrr]^{\theta\diese_\sigma}\ar[dd] & \ar[d] & & A \ar[dd] \\
& P \ar[rrd]^{\epsilon} & &  \\ 
M \ar[rrr]_{\epsilon_{\sigma}}\ar@/^/[ur]^>>{\sigma}  & & &G_{0} }
\]

One can also pull-back the Ehresmann curvature $\Omega$
via $\sigma :M\to P$ to get a 2-form $\Omega_\sigma \in \Omega^2 (M, \epsilon_\sigma^* A)$.



Let $\bis(G)$ be the group of (local) bisections of $G$.  For any $\gamma\in
C^\infty(M, \bis(G))$, $\sigma_\gamma:=\sigma \cdot \gamma$ is another
section of the principal bundle $P\to M$.
A natural question is: how  are $\theta_\sigma$ and $\Omega_\sigma$
related to  $\theta_{\sigma_\gamma}$ and $\Omega_{\sigma_\gamma}$?

\begin{Thm}\label{THM3.6}
We have the following relations:
\begin{gather}
\theta_{\sigma_\gamma} = \Ad_{\gamma\inv}  \theta_\sigma  + \gamma^{-1} \gamma' \;, \label{PhysGaugeTrasf} \\ 
\Omega_{\sigma_\gamma} = \Ad_{\gamma\inv} \Omega_{\sigma} \;. \label{Eq3.5}
\end{gather}
\end{Thm}

First we need to postulate the meaning
of the  transformation law \eqref{PhysGaugeTrasf} for the connection form.
By thinking of  $C^\infty(M, \bis(G))$ as the \emph{gauge group} $\calG$,
Equation~\eqref{PhysGaugeTrasf}
is exactly the physicists' shorthand notation for the gauge transformed 
vector potential. 

The first term is easily interpreted as an adjoint action: for any $m\in M$, $\Ad_{\gamma^{-1}(m)}$ is 
the automorphism of the Lie algebroid $A$ induced by the adjoint action of the bisection $\gamma\inv(m)$. 
The bundle map $(\Ad_{\gamma\inv}\theta_\sigma)^{\sharp}: A \to A$ over $\epsilon_{\sigma,\gamma} :=t \circ R_\gamma \circ \epsilon_{\sigma}$ 
maps $X_m\in T_m M$ to 
\[ \Ad_{\gamma\inv(m)} (\theta (X_m)) 
\in A_{t(\epsilon(\sigma(m))\cdot\gamma(m))} .\]

The second term is interpreted in terms of the Maurer Cartan form on $\bis(G)$. Let $T^{s}G = \ker s_{*}$. We define the tangent space of $\bis(G)$ at $\sigma$ as the space of maps $\Gamma(G_{0},\sigma^{*}T^{s}G)$ and 
the tangent bundle as 
\[ T \bis(G) = \{(\sigma, \calX_{\sigma}) | \sigma \in \bis(G), \calX_{\sigma} \in \Gamma(G_{0}, \sigma^{*}T^{s}G)\} \;. \]
The Maurer Cartan form $\theta_{MC}$ is defined as follows:
\begin{equation*}
\theta_{MC}: T \bis(G) \to \Gamma(A) :
(\sigma,\calX_{\sigma}) \mapsto (L_{\sigma}^{-1})_{*}\calX_{\sigma} .
\end{equation*}
Using $\gamma \in C^\infty(M, \bis(G))$ and $\epsilon_{\sigma}:M\to G_{0}$, one can define a bundle map $\Theta_{(\gamma, \epsilon_{\sigma})}: TM \to A$ over $\epsilon_{\sigma,\gamma}$
as follows:
\begin{equation*}
\Theta_{(\gamma, \epsilon_{\sigma})}: TM \to A :
X_{m} \mapsto \underbrace{((L^{-1}_{\gamma(m)})_{*} \underbrace{(\gamma_{*} X_{m})}_{T_{\gamma(m)}\bis(G)})}_{\Gamma(A)} |_{\epsilon_{\sigma,\gamma}(m)} 
.\end{equation*}
That is to say, $\Theta_{(\gamma,\epsilon_{\sigma})} = \gamma^{*}\theta_{MC}|_{\epsilon_{\sigma,\gamma}} $ is the pullback of $\theta_{MC}$ via $\gamma$ evaluated along $\epsilon_{\sigma,\gamma}$. 

\begin{proof}[Proof of Theorem~\ref{THM3.6}]
Take $\gamma\in\bis(G)$ and let $\sigma$ be a section of $\phi:P\to M$. Set $\sigma_i=\sigma$ and $\sigma_j=\sigma_i\cdot\gamma$. And let $\Phi_i$ and $\Phi_j$ denote the trivializations associated to $\sigma_i$ and $\sigma_j$.
We note that 
$\sigma_i(m)=\sigma_j(m)\cdot\delta(\sigma_j(m),\sigma_i(m))$.
We now set to derive the gauge transformation formula \eqref{PhysGaugeTrasf}, i.e. the expression for $\theta_{j} = \sigma_{j}^{*}\alpha$ in terms of $\theta_{i} = \sigma_{i}^{*}\alpha$.
Let $X_{m} \in T_{m}U$.	
\begin{equation}\label{GT1}
\begin{split}
\theta_{j}(X_{m}) =& (\sigma_{j}^{*}\alpha)(X_{m}) \\ 
=& [(\Phi_{i}^{-1} \circ \Phi_{i})^{*} \alpha](\sigma_{j,*}.X_{m}) \\ 
=& (\Phi_{i}^{*} \alpha) ((\Phi^{-1}_{i}\circ \sigma_{j})_{*}.X_{m}) 
.\end{split}
\end{equation}

But \[ \Phi_i\inv\rond\sigma_j(m)=
\Phi_i\inv\rond\Phi_j(m,e)=
(m,\delta(\sigma_i(m),\sigma_j(m))) ,\] 
and $\sigma_j(m)=\sigma_i(m)\cdot \gamma(m)_{|\epsilon(\sigma_i(m))}$ implies that 
\[ \delta(\sigma_{i}(m), \sigma_{j}(m)) 
= \delta(\sigma_{i}(m), \sigma_{i}(m)) \cdot \gamma(m)[\epsilon(\sigma_{i}(m))] 
= \gamma(m)[\epsilon_{i}(m)] .\]
To compute the derivative of $\Phi^{-1}_{i}\circ\sigma_{j}$ at a vector $X_{m}\in TM$, we pick a path $\tau\mapsto x(\tau)$ such that $\dot{x}(0)=X_{m}$.
Then \begin{align*}
(\Phi^{-1}_{i}\circ\sigma_{j})_{*} X_{m} 
=& (X_{m},\frac{d}{d\tau}\Big|_{0} \gamma(x(\tau))[\epsilon_{i}(x(\tau))]) \\
=& (X_{m},\gamma(m)_{*} \epsilon_{i,*}.X_{m} + (\gamma_{*}.X_{m})(\epsilon_{i}(m))) 
.\end{align*}
Resuming the computation \eqref{GT1}:
\begin{equation}
\label{YetAnotherNastyComp}
\begin{split}
\theta_{j}(X_{m}) 
=& (\Phi_{i}^{*}\alpha) ((\Phi^{-1}_{i}\circ \sigma_{j})_{*} X_{m}) \\ 
=& (\Phi_{i}^{*}\alpha) 
(X_{m},\gamma(m)_{*}.\epsilon_{i,*}.X_{m} + (\gamma_{*}.X_{m})(\epsilon_{i}(m)) \\
=& (\Phi_{i}^{*}\alpha) (X_{m},\gamma(m)_{*}. \epsilon_{i,*}.X_{m}) + (\Phi_{i}^{*}\alpha) (0_{m},(\gamma_{*}.X_{m})(\epsilon_{i}(m)) 
.\end{split}
\end{equation}
Notice that 
$(X_{m},\gamma(m)_{*}.\epsilon_{i,*}.X_{m}) 
\in T\epsilon_{i}^{*}\calU_{G}$ as
\begin{equation*}
s_{*}.\gamma(m)_{*}.\epsilon_{i,*}.X_{m} 
= \id_{*}.\epsilon_{i,*}.X_{m} 
=\epsilon_{i,*}.X_{m} 
\end{equation*}
and that the vector $(\gamma_{*}.X_{m})(\epsilon_{i}(m))$ is tangent to an $s$-fiber, therefore \eqref{YetAnotherNastyComp} is well defined.

The second contribution can be computed as follows:
\begin{equation*}
\begin{split}
(\Phi_{i}^{*} \alpha) 
(0_{m},(\gamma_{*}.X_{m})(\epsilon_{i}(m))) 
=& \alpha \Big(\frac{d}{d\tau} \Phi_{i}\big(m,\gamma\big(x(\tau)\big)(\epsilon_{i}(m))\Big) \\
=& \alpha \Big( \frac{d}{d\tau} \sigma_{i}(m) \cdot \gamma\big(x(\tau)\big)(\epsilon_{i}(m))\Big) 
.\end{split}
\end{equation*}
Using $x(0)=m$, $\gamma\big(x(0)\big)(\epsilon_{i}(m)) 
=\delta(\sigma_{i}(m),\sigma_{j}(m))$, 
the definition of $\alpha$ and the fact that 
$\left.\frac{d}{d\tau}\sigma_i(m)\cdot\gamma(x(\tau))[\epsilon_i(m)]\right|_0$ is vertical, one further obtains:
\begin{equation*}
\begin{split}
(\Phi_{i}^{*}\alpha) 
(0_{m},(\gamma_{*}.X_{m})(\epsilon_{i}(m)))
=&
\alpha\Big(\left.\frac{d}{d\tau}\sigma_j(m)\cdot
\delta(\sigma_j(m),\sigma_i(m))
\gamma(x(\tau))[\epsilon_i(m)]\right|_0\Big) \\ 
=&
\left.\frac{d}{d\tau}\delta(\sigma_i(m),\sigma_j(m))\inv\cdot\gamma(x(\tau))[\epsilon_i(m)] 
\right|_0 \\ 
=&
\left.\frac{d}{d\tau}
L_{\gamma(m)[\epsilon_i(m)]}\inv 
\big(\gamma(x(\tau))[\epsilon_i(m)]\big)\right|_0  
= \Theta_{(\gamma,\epsilon_{i})}.X_{m} 
.\end{split}
\end{equation*}
For the first contribution, we claim:
\[
(\Phi_{i}^{*} \alpha) (X_{m}, \gamma(m)_{*}. \epsilon_{i,*}.X_{m}) = \Ad_{\gamma^{-1}(m)}.\theta_{i}(X_{m}) \;.
\]
Indeed:
\begin{align*}
(\Phi_{i}^{*} \alpha) (X_{m}, \gamma(m)_{*}. \epsilon_{i,*}.X_{m}) 
=& \alpha\Big(\frac{d}{d\tau}\Big|_{0} \Phi_{i} \big(x(\tau), \gamma(m)(\epsilon_{i}(x(\tau))\big) \Big) \\
=& \alpha\Big( \frac{d}{d\tau}\Big|_{0} R_{\gamma(m)}\sigma_{i}(x(\tau))\Big) \;.
\end{align*}
On the other hand:
\begin{align*}
\Ad_{\gamma^{-1}(m)}.\theta_{i}(X_{m}) 
=& \Ad_{\gamma^{-1}(m)} \sigma_{i}^{*}\alpha(X_{m}) \\
=& \Ad_{\gamma^{-1}(m)} \alpha(\sigma_{i,*}.X_{m}) \\
=& \alpha( (R_{\gamma(m)})_{*}. \sigma_{i,*}.X_{m}) \;,
\end{align*}
which concludes the derivation of Equation~\eqref{PhysGaugeTrasf}.
\end{proof}

Consider now a local section $\sigma_{i}:U \to P_{|U}$, it defines uniquely a trivialization $(\epsilon_{i}^{*}\calU_{G}, \Phi_{i})$:
\[ \Phi_{i}: \epsilon_{i}^{*}\calU_{G} \to P_{|U}: (m,g) \mapsto \sigma_{i}(m)g .\]
Given $\gamma \in C^{\infty}(M,\bis(G))$, one can obtain a new section $\sigma_{j} := \sigma_{i} \star \gamma$:
\[ \sigma_{j}: m \mapsto \sigma_{i}(m) \star \gamma(m) ,\]
and the corresponding new trivialization $(\epsilon_{j}^{*}\calU_{G}, \Phi_{j})$. Each section defines a local connection form as pullback, $\theta_{i}:= \sigma_{i}^{*}\alpha$ and $\theta_{j}:= \sigma_{j}^{*}\alpha$. In this context, the gauge transformation formula is the expression for $\theta_{j}$ in terms of $\theta_{i}$. 

\begin{Cor}\label{thm.GaugeTransf}
With the notations above, 
the gauge transformation formula holds:
\begin{gather*}
\theta_{j}=\Ad_{\gamma\inv}\theta_{i}+\gamma^{-1}\gamma' \\ 
\Omega_{j}=\Ad_{\gamma\inv}\Omega_{i} \;.
\end{gather*}
\end{Cor}

\subsection{Trivial principal bundles}
We now consider trivial principal bundles. It is in this setting that our general construction reduces to the standard non linear gauge theories found in the literature \cite{bonechi.zabzine:2005a, Ikeda:1994, schaller.strobl:1994a, cattaneo.felder:2001c}. 

Let $f:M\to G_0$ and let $f^{*}\calU_{G}$ be the trivial $G$-bundle over $M$ w.r.t. $f$. As seen in Sections~\ref{connections} and \ref{subsec.Moerdijk}, we have the following

\begin{Prop}
\begin{enumerate}
\item A principal connection on $f^{*}\calU_{G}$ is determined by a bundle map
$\theta\diese: TM\to A$ over $f: M\to G_0$ satisfying the anchor condition 
$\rho\rond\theta\diese=f_*$.
\item Principal flat connections are in one-one correspondence with Lie algebroid morphisms: $TM\to A$ over $f:M\to G_0$.
\end{enumerate}
\end{Prop}





This motivated the following:

\begin{Def} By the \emph{curvature} of a bundle map 
$\theta\diese:TM\to A$, we mean the graded map 
$F_{\theta}: \Omega^{\bullet}(A) \to \Omega^{\bullet + 1}(M)$ defined by:
\[ F_{\theta} = \dd_{\mathrm{M}} \circ (\theta^{\sharp})^{*} - (\theta^{\sharp})^{*} \circ \dd_{A}  \; .\]
\end{Def}

The curvature therefore measures the failure of $\theta\diese$ being a Lie algebroid morphism. 
By Proposition~\ref{cd} and Equation~\eqref{Eq3.5}, if $\theta^{\sharp}$ satisfies the anchor condition 
$ f_*=\rho \circ \theta^\sharp$,
then $\Ad_{\gamma^{-1}}F_{\theta}= F_{\theta_{\gamma}}$.

In fact, we can prove that this relation always holds without any assumption on the anchor.

\begin{Thm}\label{thm.GaugeInvariance}
The curvature $F_{\theta}$ is always gauge invariant, i.e. 
\begin{equation}\label{GaugeInvarianceCurvature}
\Ad_{\gamma\inv} F_{\theta} = F_{\theta_{\gamma}} 
\; .\end{equation}
\end{Thm}

\begin{proof}
As an algebra, $\Omega^{\bullet}(A)$ is generated by 
$\Omega^{\leq 1}(A) = C^{\infty}(G_{0}) \oplus \Omega^{1}(A)$. Since $F_{\theta}$ is a derivation, 
it suffice to check Equation~\eqref{GaugeInvarianceCurvature} on the generators. 

We consider degree $0$ first. 
Let $f \in C^{\infty}(G_{0})$ and $X_{m} \in T_{m}M$.

First we compute $\braket{F_{\gamma}. f}{X_{m}}$.
\begin{align*}
\braket{F_{\gamma}. f}{X_{m}} 
=& \braket{\dd_{M}\theta_{\gamma}^{*}. f}{X_{m}}
-\braket{\theta_{\gamma}^{*} \dd_{A} f}{X_{m}} \\
=& \braket{\epsilon^{*}_{\gamma} \dd f}{X_{m}} 
-\rho(\theta_{\gamma}. X_{m}).f \\
=& X_{m}(f \circ \epsilon_{\sigma,\gamma}) 
-\rho(\Ad_{\gamma\inv(m)}.X_{m}).f 
-\rho( (L^{-1}_{\gamma(m)})_{*} \gamma_* . X_{m}|_{\epsilon_{\sigma,\gamma}(m)}).f 
\;.\end{align*}

On the other hand $\braket{\Ad_{\gamma^{-1}}F. f}{X_{m}}$ gives:
\begin{align*}
\braket{\Ad_{\gamma^{-1}}F. f}{X_{m}} 
=& \braket{F(\Ad_{\gamma^{-1}(m)}^{*} f)}{X_{m}} \\
=& \braket{\dd_{M}\theta^{*}(\Ad_{\gamma^{-1}(m)}^{*}.f)}{X_{m}} 
-\braket{\theta^{*}\dd_{A}\Ad^{*}_{\gamma^{-1}(m)} .f}{X_{m}} \\
=& X_m(f\rond\Ad_{\gamma\inv(m)}\rond\epsilon_\sigma) 
-\rho(\theta X_m)\big(f\rond\Ad_{\gamma\inv(m)}\big) \\ 
=& X_m((\Ad_{\gamma\inv(m)}\epsilon_\sigma)^*f) 
-\rho(\Ad_{\gamma\inv(m)}\theta.X_m).f 
\; .\end{align*}
Therefore:
\begin{multline*}
\braket{(F_{\gamma} - \Ad_{\gamma^{-1}}F). f}{X_{m}} = 
X_{m}(f \circ \epsilon_{\sigma,\gamma}) 
-X_{m}((\Ad_{\gamma^{-1}(m)}\epsilon_{\sigma})^{*}f) \\ 
-\rho( (L^{-1}_{\gamma(m)})_{*} \gamma_* . X_{m}|_{\epsilon_{\sigma,\gamma}(m)}).f 
\; .\end{multline*} 
Notice that the functions $\epsilon_{\sigma,\gamma}$ and $\Ad_{\gamma^{-1}(m)}\epsilon_{\sigma}$ are different:
\begin{gather*}
\epsilon_{\sigma,\gamma}: x \mapsto \Ad_{\gamma^{-1}(x)}\epsilon_{\sigma}(x) \, ,\\
\Ad_{\gamma^{-1}(m)}\epsilon_{\sigma}: x \mapsto \Ad_{\gamma^{-1}(m)}\epsilon_{\sigma}(x) \,.
\end{gather*}
Take a path $\tau\mapsto x_{\tau}$ in $M$ such that 
$\frac{d x_{\tau}}{d\tau}|_{0}= X_{m}$, then:
\begin{gather*}
X_{m}(f \circ \epsilon_{\sigma,\gamma})
=\frac{d}{d\tau}f(\Ad_{\gamma^{-1}(m)}
\epsilon_{\sigma}(x_{\tau}))|_{0}+\frac{d}{d\tau} f(\Ad_{\gamma^{-1}(x_{\tau})}\epsilon_{\sigma}(m))|_{0} \, ,\\
X_{m}((\Ad_{\gamma^{-1}(m)}\epsilon_{\sigma})^{*}f) 
=\frac{d}{d\tau}f(\Ad_{\gamma^{-1}(m)}
\epsilon_{\sigma}(x_{\tau}))|_{0} \;.
\end{gather*}
Hence:
\[ \braket{(F_{\gamma}-\Ad_{\gamma^-1}F).f}{X_{m}} = 
\frac{d}{d\tau} f(\Ad_{\gamma^{-1}(x_{\tau})}\epsilon_{\sigma}(m))|_{0}
-\rho( (L^{-1}_{\gamma(m)})_{*} \gamma_*. X_{m}|_{\epsilon_{\sigma,\gamma}(m)}).f \; .\]
By definition of the adjoint action on the base $G_{0}$:
\begin{align*}
\Ad_{\gamma^{-1}(x_{\tau})}\epsilon_{\sigma}(m) 
=& \target(\epsilon_{\sigma}(m) \cdot \gamma(x_{\tau})) \\
=& \target\big( R_{\gamma(x_{\tau})}(\epsilon_{\sigma}(m)) \big) \\
=& \target\big( \gamma(x_{\tau})(\epsilon_{\sigma}(m)) \big) \;.
\end{align*}
On the other hand:
\begin{align*}
\rho( (L^{-1}_{\gamma(m)})_{*} \gamma_* . X_{m}|_{\epsilon_{\sigma,\gamma}(m)}) 
=& \frac{d}{d\tau}\Big|_{0} (\target\circ L^{-1}_{\gamma(m)})\big(\gamma(x_{\tau})
(\epsilon_{\sigma}(m))\big) \\
=& \frac{d}{d\tau}\Big|_{0} \target \big(\gamma(x_{\tau})(\epsilon_{\sigma}(m))\big) \;,
\end{align*}
thus:
\[ \rho( (L^{-1}_{\gamma(m)})_{*} \gamma_* . X_{m}|_{\epsilon_{\sigma,\gamma}(m)}).f = \frac{d}{d\tau}\Big|_{0} f(\target \big(\gamma(x_{\tau}) (\epsilon_{\sigma}(m))\big)) = \frac{d}{d\tau}\Big|_{0} f(\Ad_{\gamma^{-1}(x_{\tau})}\epsilon_{\sigma}(m)) 
\; .\]
Since $X_{m}$ is arbitrary, we conclude that:
\[ (F_{\gamma} - \Ad_{\gamma^{-1}}F). f = 0 \; .\]

To check gauge invariance in degree $1$ we choose 
$\lambda \in \Omega^{1}A$. The following is just an unravelling of definitions: 
\begin{align*}
(F_{\gamma}-\Ad_{\gamma^{-1}}F)\lambda 
=& -(\Ad_{\gamma\inv}\theta^{*})^*(\dd_{A}\lambda) 
-\Theta_{(\gamma,\epsilon_{\sigma})}(\dd_{A}\lambda) 
+\dd_{M}((\Ad_{\gamma\inv}\theta)^{*}.\lambda) \\
& +\dd_{M}(\Theta_{(\gamma,\epsilon_{\sigma})}^*.\lambda)   
+\theta^{*}\circ \dd_{A}(\Ad^{*}_{\gamma^{-1}}.\lambda) 
-\dd_{M}\circ\theta^{*}(\Ad^{*}_{\gamma^{-1}}.\lambda) \\ 
=& \Theta^{*}_{(\gamma,\epsilon_{\sigma})}(\dd_{A}\lambda) 
-\dd_{M}(\Theta^{*}_{(\gamma,\epsilon_{\sigma})}.\lambda) 
\;. \end{align*}
We used the fact the $\Ad_{\gamma^{-1}(m)}$ is an algebroid morphism for all $m$ in $M$. Recall that  $\Theta_{(\gamma,\epsilon_{\sigma})}$ is the pullback of the Maurer Cartan form on $\bis(G)$ via $\gamma:M\to\bis(G)$ evaluated along $\epsilon_{\sigma,\gamma}:M\to G_{0}$. We now show that 
\begin{equation}\label{EQ3.12}
\braket{\Theta^{*}_{(\gamma,\epsilon_{\sigma})}(\dd_{A}\lambda)
-\dd_{M}(\Theta^{*}_{(\gamma,\epsilon_{\sigma})}.\lambda)}
{X_{m}\wedge Y_{m}} = 0 
,\end{equation}
for all $X,Y\in\mathfrak{X}(M)$ and $m\in M$.
Using the definition of $\Theta_{(\gamma,\epsilon_{\sigma})}$:
\[ \Theta_{(\gamma,\epsilon_{\sigma})}: X_{m} \mapsto (L^{-1}_{\gamma(m)})^{*}\gamma_{*}. X_{m}|_{\epsilon_{\sigma,\gamma}(m)} \; ,\]
we compute separately the two terms.
The first one yields:
\begin{multline*}
\braket{\Theta^{*}_{(\gamma,\epsilon_{\sigma})}(\dd_{A} \lambda)}{X_{m}\wedge Y_{m}} = \\ \rho(\Theta_{(\gamma,\epsilon_{\sigma})}.X_{m})\braket{\lambda}{(L^{-1}_{\gamma(m)})^{*}\gamma_{*}.Y_{m})} 
-\rho(\Theta_{(\gamma,\epsilon_{\sigma})}.Y_{m})\braket{\lambda}{(L^{-1}_{\gamma(m)})^{*}\gamma_{*}.X_{m})} \\
-\braket{\lambda}{\Lie{(L^{-1}_{\gamma(m)})^{*}\gamma_{*}.X_{m}}{(L^{-1}_{\gamma(m)})^{*}\gamma_{*}.Y_{m}}}_{\epsilon_{\sigma,\gamma}(m)}.\end{multline*}
The second term reads:
\begin{multline*}
\braket{\dd_{M}(\Theta^{*}_{(\gamma,\epsilon_{\sigma})}.\lambda)}
{X_{m}\wedge Y_{m}}
=\braket{\dd_{M}(\Theta^{*}_{(\gamma,\epsilon_{\sigma})}.\lambda)}{X\wedge Y}_{m}
=X_{m}\braket{\Theta^{*}_{(\gamma,\epsilon_{\sigma})}.\lambda}{Y} \\ 
-Y_{m}\braket{\Theta^{*}_{(\gamma,\epsilon_{\sigma})}.\lambda}{X}
-\braket{\lambda}{(L^{-1}_{\gamma(m)})^{*}\gamma_{*}.\Lie{X}{Y}_{m}}_{\epsilon_{\sigma,\gamma}(m)} 
\; .\end{multline*}
Since solutions to the Maurer Cartan equations are preserved by pullbacks, for every $m\in M$ we have:
\begin{multline*}
(L^{-1}_{\gamma(m)})^{*}\gamma_{*}.\Lie{X}{Y}_{m} 
-\Lie{(L^{-1}_{\gamma(m)})^{*}\gamma_{*}.X_{m}}
{(L^{-1}_{\gamma(m)})^{*}\gamma_{*}.Y_{m}} \\
= X_{m}((L^{-1}_{\gamma})^{*}\gamma_{*}.Y) 
-Y_{m}((L^{-1}_{\gamma})^{*}\gamma_{*}.X) 
\; .\end{multline*}
Therefore, the L.H.S. of Equation~\eqref{EQ3.12} is the sum of the two triples 
\[ \rho(\Theta_{(\gamma,\epsilon_{\sigma})}.X_{m})
\braket{\lambda}{(L^{-1}_{\gamma(m)})^{*}\gamma_{*}.Y_{m})}
- X_{m}\braket{\Theta^{*}_{(\gamma,\epsilon_{\sigma})}.\lambda}{Y} 
+ \braket{\lambda}{X((L^{-1}_{\gamma})^{*} \gamma_{*}.Y)}_{\epsilon_{\sigma,\gamma}(m)} \]
and
\[ \rho(\Theta_{(\gamma,\epsilon_{\sigma})}.Y_{m}) \braket{\lambda}{(L^{-1}_{\gamma(m)})^{*}\gamma_{*}.X_{m})}
- Y_{m} \braket{\Theta^{*}_{(\gamma,\epsilon_{\sigma})}.\lambda}{X} 
+ \braket{\lambda}{Y((L^{-1}_{\gamma})^{*} \gamma_{*}.X)}_{\epsilon_{\sigma,\gamma}(m)} .\]
In the first triple, the first term is the derivative with respect to 
\[ \rho(\Theta_{(\gamma,\epsilon_{\sigma})}.X_{m}) =\rho((L^{-1}_{\gamma(m)})^{*}\gamma_{*}.X_{m}|_{\epsilon_{\sigma,\gamma}(m)}) \] 
of the function
\begin{equation*}
G_{0} \in g \mapsto \braket{\lambda}{(L^{-1}_{\gamma(m)})^{*}\gamma_{*}.Y_{m})}_{g}\;,
\end{equation*}
the second term is the derivative with respect to $X_{m}$ of the function
\begin{equation*}
M \in x \mapsto \braket{\lambda}{(L^{-1}_{\gamma(x)})^{*}\gamma_{*}.Y_{x})}_{\epsilon_{\sigma,\gamma}(x)}\;
\end{equation*}	
and the third term is the derivative with respect to $X_{m}$ of the function:
\begin{equation*}
M \in x \mapsto \braket{\lambda}{(L^{-1}_{\gamma(x)})^{*}\gamma_{*}.Y_{x})}_{\epsilon_{\sigma,\gamma}(m)}\;.
\end{equation*}
Therefore they cancel by Leibniz rule. The same argument applies to the second triple and the theorem is proven.
\end{proof}

By $\mathcal{A}_0$, we denote the space of Lie algebroid morphisms $TM\to A$ on which $\mathcal{G}=\cinf{M,\bis(G)}$ acts naturally.
Therefore the \emph{moduli space of flat connections} $\calA_{0}/\calG$ is well defined.
This construction was studied in detail in \cite{bonechi.zabzine.2007}. 

A special case of this construction gives a gauge 
theoretical interpretation of the topological groupoid $G(A)$, which integrates $A$. In that case, $M$ is the interval $I=[0,1]$. 
The original approach involves the Weinstein groupoid of $A$ which is the space of algebroid morphisms from $TI$ to $A$ modulo algebroid homotopies (see \cite{crainic.fernandes:2003} for details). It is a proposition in \cite{bojowald.kotov.strobl:2005a} that two Lie algebroid morphisms are homotopic if and only if they are connected by the flow of an infinitesimal gauge transformation. Therefore the Weinstein groupoid of $A$ can be interpreted as the moduli space of flat connections on the interval $I$.






\bibliographystyle{amsalpha}
\bibliography{MyBib}

\providecommand{\bysame}{\leavevmode\hbox to3em{\hrulefill}\thinspace}
\providecommand{\MR}{\relax\ifhmode\unskip\space\fi MR }
\providecommand{\MRhref}[2]{%
  \href{http://www.ams.org/mathscinet-getitem?mr=#1}{#2}
}
\providecommand{\href}[2]{#2}
\begin{thebibliography}{CdSW99}

\bibitem[BKS05]{bojowald.kotov.strobl:2005a}
M.~Bojowald, A.~Kotov, and T.~Strobl, \emph{Lie algebroid morphisms, {P}oisson
  sigma models, and off-shell closed gauge symmetries}, J. Geom. Phys.
  \textbf{54} (2005), 400--426.

\bibitem[BZ05]{bonechi.zabzine:2005a}
F.~Bonechi and M.~Zabzine, \emph{Poisson sigma model over group manifolds}, J.
  Geom. Phys. \textbf{54} (2005), 173--196.

\bibitem[BZ07]{bonechi.zabzine.2007}
Francesco Bonechi and Maxim Zabzine, \emph{Lie algebroids, {L}ie groupoids and
  {TFT}}, J. Geom. Phys. \textbf{57} (2007), no.~3, 731--744.

\bibitem[CdSW99]{cannasdasilva.weinstein:1999a}
A.~Cannas~da Silva and A.~Weinstein, \emph{Geometric models for noncommutative
  algebras}, Berkeley Mathematics Lecture Notes, AMS, 1999.

\bibitem[CF00]{cattaneo.felder:2000a}
A.~Cattaneo and G.~Felder, \emph{A path integral approach to the kontsevich
  quantization formula}, Commun. Math. Phys. \textbf{212} (2000), 591--611.

\bibitem[CF01]{cattaneo.felder:2001c}
A.~S. Cattaneo and G.~Fedler, \emph{Poisson sigma models and symplectic
  groupoids}, Quantization of singular symplectic quotients (N.~P. Landsman,
  M.~Pflaum, and M.~Schlichenmaier, eds.), Birkh{\"a}user, Basel, Boston,
  Berlin, 2001, pp.~61--93.

\bibitem[CF03]{crainic.fernandes:2003}
Marius Crainic and Rui~Loja Fernandes, \emph{Integrability of {L}ie brackets},
  Ann. of Math. (2) \textbf{157} (2003), no.~2, 575--620.

\bibitem[Ike94]{Ikeda:1994}
Noriaki Ikeda, \emph{Two-dimensional gravity and nonlinear gauge theory}, Ann.
  of Phys. \textbf{235} (1994), 435--464.

\bibitem[Kon03]{kontsevich:2003a}
M.~Kontsevich, \emph{Deformation quantization of {P}oisson manifolds}, Lett.
  Math. Phys. \textbf{66} (2003), 157--216.

\bibitem[Mac05]{mackenzie:2005a}
K.~C.~H. Mackenzie, \emph{General theory of {L}ie groupoids and {L}ie
  algebroids}, London Mathematical Society Lecture Note Series, vol. 213,
  Cambridge University Press, Cambridge, UK, 2005.

\bibitem[MM03]{moerdijk.mrcun:2003a}
I.~Moerdijk and J.~Mr{\v{c}}un, \emph{Introduction to foliations and {L}ie
  groupoids}, Cambridge studies in advanced mathematics, no.~91, Cambridge
  University Press, Cambridge, UK, 2003.

\bibitem[Ros04]{rossi:2004b}
C.~A. Rossi, \emph{The division map of principal bundles with groupoid
  structure and generalized gauge transformations}, arXiv:math/0401182v2, 2004.

\bibitem[SS94]{schaller.strobl:1994a}
P.~Schaller and T.~Strobl, \emph{Poisson structure induced (topological) field
  theories}, Mod. Phys. Lett. A \textbf{9} (1994), no.~33, 3129--3136.

\bibitem[Str04]{Strobl:2004}
Thomas Strobl, \emph{Algebroid {Y}ang-{M}ills theories}, Phys. Rev. Lett.
  \textbf{93} (2004), no.~21, 211601--4.

\end{thebibliography}

\end{document}